\documentclass{amsart}

\newtheorem{theorem}{Theorem}[section]
\newtheorem{lemma}[theorem]{Lemma}

\theoremstyle{definition}
\newtheorem{definition}[theorem]{Definition}
\newtheorem{example}[theorem]{Example}
\newtheorem{note}[theorem]{Note}

\theoremstyle{remark}
\newtheorem{remark}[theorem]{Remark}

\numberwithin{equation}{section}

\begin{document}

\title{ bounded variation on the Sierpi\'nski Gasket}

\thanks{}

\author{S. Verma}
\address{Department of Mathematics, IIT Delhi, New Delhi, India 110016 }
\email{saurabh331146@gmail.com}

\author{A. Sahu}
\address{Department of Mathematics, IIT Guwahati, Guwahati, India 781039}


\email{sahu.abhilash16@gmail.com}



 

\begin{abstract}
 Under certain continuity conditions, we estimate upper and lower box dimension of graph of a function defined on the Sierpi\'nski gasket. We also give an upper bound for Hausdorff dimension and box dimension of graph of function having finite energy. Further, we introduce two sets of definitions of bounded variation for a function defined on the Sierpi\'nski gasket. We show that fractal dimension of graph of a continuous function of bounded variation is $\frac{\log 3}{\log 2}.$ We also prove that the class of all bounded variation functions is closed under arithmetic operations. Furthermore, we show that every function of bounded variation is continuous almost everywhere in the sense of $\frac{\log3}{\log2}-$dimensional Hausdorff measure.

\end{abstract}

\maketitle



\section{INTRODUCTION}

The classical notion of  \emph{bounded variation} was introduced by Jordan \cite{Jordan} for a real-valued function on a closed bounded interval $[a,b]$ in $\mathbb{R}$.

We write several properties of a function which is of bounded variation on $[a,b]$ briefly here, but refer the reader to \cite{Gordon}.
\begin{definition}
Let $ f:[a,b] \rightarrow \mathbb{R}$ be a function. For each partition $ P: a=t_0<t_1<t_2 < \dots <t_n =b$ of the interval $[a,b], $ we define $$V(f,[a,b])= \sup_P \sum_{i=1}^{n} |f(t_i)-f(t_{i-1})|,$$ where the supremum is taken over all partitions $P$ of the interval $[a,b].$\\ If $V(f,[a,b]) < \infty,$ we say that $f$ is of bounded variation. 
\end{definition}
\begin{definition}
Let $f :[a,b] \to \mathbb{R}$ be a continuous function. The arc length of the curve $y=f(x)$ on the interval $[a,b]$ is defined by $L=\sup(S)$, where $$ S= \Big\{ \sum_{i=1}^{n}\sqrt{(x_i-x_{i-1})^2+(f(x_i)-f(x_{i-1}))^2} : \{x_i\}_{i=0}^n~ \text {is a partition of}~ [a,b] \Big\}.$$
If $f$ has a continuous derivative on $[a,b]$, then the arc length $L= \int_{a}^{b} \sqrt{1+(f'(x))^2}~ \mathrm{d}x.$
If $S$ is unbounded, then $f$ is said to have infinite length on the given interval.
\end{definition}

We list up certain properties of a class of bounded variation functions. 
\begin{itemize}\label{ET1}
\item The arc length of a curve $y=f(x)$ is finite if and only if $f$ is of bounded variation on $[a,b].$
\item A function $f$ is of bounded variation on an interval $[a,b]$ if and only if it can be decomposed as a difference of two monotonic increasing functions.
\item If $f$ is of bounded variation on $[a,b],$ then $f$ is differentiable almost everywhere and discontinuous at most at a countable number of points.
\end{itemize}

Let us recall the following theorem due to Liang \cite{Liang} which relate fractal dimension and bounded variation.
\begin{theorem} \label{ET4}
If $f :[a,b] \to \mathbb{R}$ is continuous and of bounded variation then $\dim_H(G_f)=\dim_B(G_f)=1.$
 \end{theorem}

Recently, using $L^1$ Korevaar-Schoen class at the critical exponent, a notion of bounded variation on fractal domain is introduced in \cite{AT}. 
Further, they study certain properties of the class of bounded variation such as locality property, co-area estimate and Sobolev type inequalities. They also prove that functions of bounded variation induce Radon measures on the domain. Now, We turn our attention to \cite[Conjecture $5.3$]{AT}, which conveys that non-constant continuous (or smooth) functions of bounded variation on the Sierpi\'nski gasket can not exist, this is absurd. However, our definition of bounded variation on SG includes all Lipschitz continuous functions. 
\par

In this paper, we introduce two sets of new definitions of bounded variation for functions defined on the Sierpi\'nski gasket. Further, we study similar properties of such functions as above. Our approach is different from that of \cite{AT} in the sense that we are focused on oscillation of the function. We hope that our notion will find further applications in the research area related to Fourier series, calculus of variations, Koksma-Hlawka type inequalities on the Sierpi\'nski gasket. The reader is encouraged to see \cite{APST,BCT,VVC1,VVC2,VVC3} and references therein for recent works on bounded variation functions defined on non-fractal domains.

\section{SETUP}
\subsection{Code space}
For this part, we refer the reader to \cite{MFB}.
Let $(X, d)$ be a complete metric space. Let $\{f_1,f_2,\dots , f_N\}$ be a finite sequence of contraction maps, $f_n: X \rightarrow X,$ for $n=1,2, \dots , N.$ Then $ \mathcal{F}:=\{X; f_1,f_2, \dots ,f_N\}$ is called a hyperbolic iterated function system or, briefly, an IFS.
A map $f_n: X \rightarrow X$ is contraction when there is a number $0 \le c_n < 1$ such that $ d(f_n(x),f_n(y)) \le c_n d(x,y)$ for all $x,y \in X.$ The number $c_n$ is called a contraction factor for $f_n$ and the number $ c_{\mathcal{F}}= \max\{c_1,c_2, \dots , c_N\}$ is called a contraction factor for $\mathcal{F}.$ \\
Let $ \Omega$ denote the set of all infinite sequences of symbols $\{\sigma_k\}_{k=1}^{\infty} $ belonging to the alphabet $\{1,2,\dots ,N\}.$ We write $ \sigma= \sigma_1\sigma_2\sigma_3\dots  \in \Omega$ to denote a typical element of $\Omega,$ and we write $\sigma_k$ to denote the $k$th element of $\sigma \in \Omega.$ Then $(\Omega, d_{\Omega})$ is a compact metric space, where the metric $d_{\Omega}$ is defined by $d_{\Omega}(\sigma, \omega)= 0$ when $\sigma= \omega$ and $d_{\Omega}(\sigma, \omega)= 2^{-k}$ when $k$ is the least index for which $\sigma_k \ne \omega_k.$ We call $\Omega$ the code space associated with the IFS $\mathcal{F}.$\\

Let $\sigma \in \Omega $ and $x \in X.$ Then, using the contractivity of $\mathcal{F},$ it is straightforward to prove that $$ \phi_{\mathcal{F}} (\sigma) := \lim_{k \rightarrow \infty} f_{\sigma_1} \circ f_{\sigma_2} \circ \dots \circ f_{\sigma_k}(x)$$
exists, is independent of $x,$ and depends continuously on $\sigma.$ Furthermore, the convergence to the limit is uniform in $x,$ for $x$ in any compact subset of $X.$ Let $A_{\mathcal{F}}=\{\phi_{\mathcal{F}}(\sigma): \sigma \in \Omega \}.$ Then $A_{\mathcal{F}} \subset X$ is called the attractor of $\mathcal{F}.$ The continuous function $$ \phi: \Omega \rightarrow A_{\mathcal{F}}$$ is called the address function of $\mathcal{F}.$ We call $ \phi_{\mathcal{F}}^{-1}(\{x\})=\{\sigma \in \Omega:\phi_{\mathcal{F}}(\sigma)=x\} $ the set of addresses of the point $x \in A_{\mathcal{F}}.$ \\
We order the elements of $\Omega$ according to $$\sigma \prec \omega~ \text{if and only if}~ \sigma_k < \omega_k,$$ where $k$ is the least index for which $\sigma_k \ne \omega_k .$ We observe that all elements of $\Omega $ are less than or equal to $ \overline{N}= NNN\dots $ and greater than or equal to $ \overline{1}= 111\dots .$ Note that $\phi_{\mathcal{F}}^{-1} (\{x\})$ contains a unique largest element.
Let $\mathcal{F}$ be a hyperbolic IFS with attractor $A_{\mathcal{F}}$ and address function $\phi_{\mathcal{F}}: \Omega \rightarrow A_{\mathcal{F}}.$ Let $$ \tau_{\mathcal{F}}(x)=\max\{\sigma \in \Omega : \phi_{\mathcal{F}}(\sigma)=x\}$$
for all $x \in A_{\mathcal{F}}.$ Then $$ \Omega_{\mathcal{F}}:=\{ \tau_{\mathcal{F}}(x): x \in A_{\mathcal{F}}\} $$ is called the tops code space and $$ \tau_{\mathcal{F}}: A_{\mathcal{F}} \rightarrow \Omega_{\mathcal{F}}$$ is called the tops function corresponding to the IFS $\mathcal{F}.$ It can be seen that the tops function $ \tau_{\mathcal{F}}: A_{\mathcal{F}} \rightarrow \Omega_{\mathcal{F}}$ is one-one and onto.
For $x,y \in A_{\mathcal{F}},$ we define $w^x:=\tau_{\mathcal{F}}(x)$ and $w^y:=\tau_{\mathcal{F}}(y).$ Without loss of generality we assume $w^x \prec w^y.$ We define $\Omega_{[x,y]}= \{\sigma \in \Omega_{\mathcal{F}}: \tau_{\mathcal{F}}(x) \prec \sigma \prec \tau_{\mathcal{F}}(y)\} \cup \{w^x,w^y\}.$ 
\subsection{Sierpi\'nski gasket}
 The reader is encouraged to consult \cite{Strichartz, Kigami} for a detailed study on the Sierpi\'nski gasket.
First we recall a well-known construction of the Sierpi\'nski gasket by Iterated Function System. Let us consider three points $q_i~ (i=1,2,3),$ in $\mathbb{R}^2$, which are at equidistant from each other. Corresponding to each of these points, we define maps $u_i:\mathbb{R}^2 \to \mathbb{R}^2$ by $u_i(x)=(x+q_i)/2.$ The invariant set of IFS $\{u_i: i=1,2,3\}$, is called the Sierpi\'nski gasket ($SG$ for short), that is, $$SG= \cup_{i=1}^3 u_i(SG).$$
For $n \in \mathbb{N},$ we denote the collection of all words with length $n$ by $\{1,2,3\}^n,$ that is, if $w \in \{1,2,3\}^n$ then $w=w_1,w_2 \dots w_n$ where $w_i \in \{1,2,3\}.$ We define, for $w \in \{1,2,3\}^n$, $$u_w= u_{w_1} \circ u_{w_2} \circ \dots \circ u_{w_n}~~ \text{and} ~q_w= u_{w_1w_2 \dots w_{n-1}}(q_{w_n}).$$
Define $V_0=\{q_1,q_2,q_3\}.$ We call $V_0$ the set of vertices of $SG.$ For any positive integer $n,$ we define $V_n$ to be the union of all $u_w(V_0)$ with $w \in \{1,2,3\}^n.$ Define $V_*= \cup_{n=1}^{\infty}V_n.$
 We define $ \Gamma_0$ to be the complete graph on the vertex set $V_0.$ Having constructed graph $\Gamma_{m-1}$ with vertex set $V_{m-1}$ for some $m \ge 1,$ we define the graph $\Gamma_m$ on $V_m$ as follows: for any $x,y \in V_m,$ the edge relation $ x \sim_m y$ to hold if and only if $x=u_i(x'),y=u_i(y')$ with $x' \sim_{m-1} y'$ and $i \in \{1,2,3\}.$ Equivalently, $x \sim_m y$ if and only if there exists $\omega \in \{1,2,3\}^n$ such that $x,y \in u_{\omega}(V_0).$ For $m=0,1,2, \dots ,$ we define graph energies $E_m$ on $\Gamma_m$ by $$ E_m(f):= \Big(\frac{5}{3} \Big)^m \sum_{x \sim_m y}(f(x)-f(y))^2.$$
It is well known that the graph energy sequence $\{E_m\}$ defined as above satisfies $ E_{m-1}(f)= \min E_m(\tilde{f}),$ where the minimum is taken over all $\tilde{f}$ satisfy $ \tilde{f}|_{V_{m-1}}=f$ for any $f: V_* \rightarrow \mathbb{R}$ and for any $m \ge 1.$ By above, for each function $f$ on $V_*$, sequence $\{E_m(f)\}_{m=0}^{\infty}$ is increasing. We call $$ E(f):= \lim_{m \rightarrow \infty}E_m(f)$$ the energy of $f$ on $V_*.$ Furthermore, we say $f$ has finite energy if $E(f) < +\infty.$\\
It is well known that any function $f$ with $E(f) < +\infty$ is uniformly continuous on $V_*.$ Thus $f$ can be uniquely extended to a continuous function on $SG$ since $V_*$ is dense in $SG.$\\
We call $f$ a harmonic function on $SG$ if $E_{m-1}(f)=E_m(f)$ for all $m \ge 1.$ The following property is the well-known $``\frac{1}{5}-\frac{2}{5}"$ rule for harmonic functions.
\begin{lemma}
Let $h$ be a harmonic function on $SG.$ Let $(i,j,k)$ be a permutation of $(1,2,3)$. Then $$h(q_{ij})= \frac{2}{5}h(q_i)+\frac{2}{5}h(q_j)+\frac{1}{5}h(q_k).$$ Generally, for any $w=w_1\dots w_m \in \{1,2,3\}^m,$ we have $$h(q_{wij})= \frac{2}{5}h(q_{wi})+\frac{2}{5}h(q_{wj})+\frac{1}{5}h(q_{wk}),$$ where $wi, wj, wk $ and $wij$ to be the word $w_1\dots w_m i,~w_1\dots w_m j,~ w_1\dots w_m k $ and $w_1\dots w_m ij,$ respectively.
\end{lemma}
From the above lemma, we observe that a harmonic function $h$ is determined by its values on $V_0.$ We also obtain that $h$ has the following min-max property:$$ \min_{x \in V_0} h(x) \le h(y)\le \max_{x \in V_0}h(x), ~~\text{for any } y \in SG.$$
In particular, $h$ is constant on SG if it is constant on $V_0.$
\subsection{Fractal dimensions}
We shall summarize two notions of fractal dimension briefly here, but refer the reader to \cite{Fal}.
\begin{definition}
 For a non-empty subset $U \subset \mathbb{R}^n,$ the diameter of $U$ is defined as $$|U|=\sup\big\{\|x-y\|_2: x,y \in U\big\},$$ where $\|.\|_2$ denotes the Euclidean norm. Let $F$ be a subset of $\mathbb{R}^n$ and $s$ a non-negative real number, the $s$-dimensional Hausdorff measure of $F$ is defined as $$\mathcal{H}^s(F)=\lim_{\delta \rightarrow 0^+} \Big[ \inf \Big\{\sum_{i=1}^{\infty}|U_i|^s:  F \subseteq \cup U_i, ~~ |U_i|< \delta \Big\}\Big] .$$ 
\end{definition}
 \begin{definition}
 Let $F \subseteq \mathbb{R}^n$ and $s \ge 0.$ The Hausdorff dimension of $F$ is defined as $$ \dim_H(F)=\inf\{s:\mathcal{H}^s(F)=0\}=\sup\{s:\mathcal{H}^s(F)=\infty\}.$$
 \end{definition}

 \begin{definition} Let $F$ be a nonempty bounded subset of $\mathbb{R}^n$ and let $N_{\delta}(F)$ be the smallest number of sets of diameter at most $\delta$ which can cover $F.$ The lower box dimension and upper box dimension of $F$ respectively are defined as $$\underline{\dim}_B(F)=\varliminf_{\delta \rightarrow 0^+} \frac{\log N_{\delta}(F)}{- \log \delta},$$
 and
 $$\overline{\dim}_B(F)=\varlimsup_{\delta \rightarrow 0^+} \frac{\log N_{\delta}(F)}{- \log \delta}.$$
 If the above two are equal, we call the common value as the box dimension of $F,$ that is,  $$\dim_B(F)=\lim_{\delta \rightarrow 0^+} \frac{\log N_{\delta}(F)}{- \log \delta}.$$
 \end{definition}
  In the sequel, we shall use the following result of \cite{Fal}, which reveals a fundamental property of the Hausdorff dimension and box dimension.
  \begin{theorem} \label{BR1}
  Let $A \subseteq \mathbb{R}^n$ and $f:A \to \mathbb{R}^m$ a Lipschitz map. Then $\dim_{H}\big( f(A) \big) \le \dim_{H}(A)$ and $\dim_{B}\big( f(A) \big) \le \dim_{B}(A).$
  \end{theorem}
 
\section{Graph of function on the Sierpi\'nski gasket}
\begin{lemma}\label{BVSGl1}
 If $f :SG \rightarrow \mathbb{R}$ is continuous on $SG,$ then $\dim_H (G_f) \ge \frac{\log 3}{\log 2}.$
\end{lemma}
\begin{proof}
In view of Theorem \ref{BR1}, a Lipschitz onto map $\mathcal{T}_f : G_f \rightarrow SG $ defined by $\mathcal{T}_f((t,f(t)))=t$, produces the result.
\end{proof}

Maximum range of  $f$ over a part $u_w(SG)$ is defined by
$$ R_f[u_w(SG)]= \sup_{x,y \in u_w(SG)} |f(x)-f(y)|.$$
\begin{lemma}\label{BVSGl2}
Let $f :SG \rightarrow \mathbb{R}$ be continuous. Suppose that $ \delta = \frac{1}{2^n}$ for some $n \in \mathbb{N}.$ If $N_{\delta}(SG)$ denotes the number of $\delta-$cubes that intersect graph of $f,$ then $$ 2^n \sum_{w \in \{1,2,3\}^n} R_f[u_w(SG)] \le N_{\delta}(SG) \le 2.3^n + 2^n \sum_{w \in \{1,2,3\}^n} R_f[u_w(SG)].$$
\end{lemma}
\begin{proof}
By using continuity of $f$, the number of cubes of side length $\delta$ in the part above $u_w(SG)$ that intersect the graph of $f$ is at least $\frac{R_f[u_w(SG)]}{ \delta} $ and at most  $2+ \frac{R_f[u_w(SG)]}{ \delta} .$ Now, summing over all such parts yields $ 2^n \sum_{w \in \{1,2,3\}^n} R_f[u_w(SG)] $ lower bound and $2.3^n+ 2^n \sum_{w \in \{1,2,3\}^n} R_f[u_w(SG)] $ upper bound of $N_{\delta}(SG).$

\end{proof}

\begin{theorem} \label{BVSGth1}
Let $f:SG \rightarrow \mathbb{R}$ be a continuous function.
\begin{enumerate}
\item Suppose $|f(x)-f(y)| \le c ~\|x-y\|^s,~~\forall ~~x, y \in SG,$ where $c>0$ and $0 \le s \le 1.$ Then $\dim_H (G_f) \le \overline{\dim}_B (G_f )\le 1-s+ \frac{\log3}{\log2}.$ The conclusion remains true if H\"{o}lder condition holds when $\|x-y\| < \delta $ for some $\delta >0.$
\item Suppose that there are numbers $c>0, \delta_0 >0$ and $0 \le s \le 1$ with the following property: for each $y \in SG$ and $ 0< \delta <\delta_0$ there exists $x$ such that $\|x-y\| \le \delta$ and $$ |f(x)-f(y)| \ge c \delta^s.$$ Then $ \underline{\dim}_B (G_f) \ge 1- s +\frac{\log3}{\log2}.$
\end{enumerate}

\end{theorem}
\begin{proof}
\begin{enumerate}
\item
Since $f$ satisfies H\"{o}lder condition, we have $ R_f[u_w(SG)] \le  \frac{c}{2^{ns}}.$ From the previous lemma, we obtain $N_{\delta}(SG) \le 2.3^n + c2^{n(1-s)}3^n .$
Upper box-dimension of $G_f$ can be estimated in the following way $$ \overline{\lim}_{\delta \rightarrow 0} \frac{\log N_{\delta}(SG)}{-\log \delta} \le \lim_{n \rightarrow \infty}\frac{\log(2.3^n + c2^{n(1-s)}3^n)}{\log 2^n}.$$
Which produces $ \overline{\lim}_{\delta \rightarrow 0} \frac{\log N_{\delta}(SG)}{-\log \delta} \le 1-s +\frac{\log3}{\log2}.$ 
\item  Using hypothesis, we get $R_f[u_w(SG)] \ge c \delta ^s= \frac{c}{2^{ns}}.$ 
The previous lemma yields $ N_{\delta}(SG) \ge  c2^{n(1-s)}3^n .$ We estimate lower box dimension of $G_f$ in the similar manner and arrive at $ \underline{\dim}_B G_f \ge 1- s +\frac{\log3}{\log2}.$
\end{enumerate}

\end{proof}

\begin{theorem}[\cite{Fukushima}] \label{BVSGth2}
Let $f:V_* \rightarrow \mathbb{R}$ be a function. Then 
$$\sup_{x,y \in V_*} \frac{|f(x)-f(y)|}{\|x-y\|^{\sigma}} \le 9~ \sqrt{E(f)}$$ where $\sigma = \frac{\log (5/3)}{2 \log 2}.$
\end{theorem}
\begin{theorem}
If $f:SG \rightarrow \mathbb{R} $ is a continuous function and $E(f) < \infty$ then $$\frac{\log 3}{\log 2} \le \dim_H(G_f) \le \overline{\dim}_B(G_f) \le \frac{\log(108/5)}{2 \log 2}.$$
\end{theorem}
\begin{proof}
Proof follows from Theorems \ref{BVSGl1} and \ref{BVSGth2}.
\end{proof}
The sapce of all biharmonic functions forms a vector space of dimension $6.$ For more details about biharmonic functions, the reader is referred to \cite{Strichartz3,Strichartz2}. Now, we will discuss about box dimension of biharmonic functions on the SG.
\begin{theorem}
Let $f$ be a biharmonic function on SG. Then $$\frac{\log 3}{\log 2} \le \dim_H(G_f) \le \overline{\dim}_B(G_f) \le \frac{\log(18 / 5)
}{\log 2}.$$
\end{theorem}
\begin{proof}
By definition of biharmonic function, we note that 
\begin{equation}\label{biharmonic}
\begin{aligned}
& \Delta f(q_{wij})=  \frac{2}{5} \Delta f(q_{wi}) + \frac{2}{5} \Delta f(q_{wj})+\frac{1}{5} \Delta f(q_{wk}),\\
& 4 f(q_{wij})- f(q_{wi})-f(q_{wj})-f(q_{wjk})-f(q_{wki})= \frac{2}{3} 5^{-m} \Delta f(q_{wij}),\\
 & 4 f(q_{wjk})-f(q_{wj})- f(q_{wk})-f(q_{wij})-f(q_{wki})= \frac{2}{3} 5^{-m} \Delta f(q_{wjk}),\\
 & 4 f(q_{wki})- f(q_{wk})-f(q_{wi})-f(q_{wij})-f(q_{wjk})= \frac{2}{3} 5^{-m} \Delta f(q_{wki}).
\end{aligned}
\end{equation} 
From the above equations, one deduces
\begin{equation*}
\begin{aligned}
 f(q_{wij})  = &\frac{f(q_{wk})+2f(q_{wi})+2f(q_{wj})}{5} \\ & +\frac{1}{3} 5^{-m} \Big(\frac{3}{5}\Delta f(q_{wij})+\frac{1}{5} \Delta f(q_{wki}) + \frac{1}{5} \Delta f(q_{wjk})\Big).
 \end{aligned}
 \end{equation*}
Using Equation \ref{biharmonic},
\begin{equation*}
\begin{aligned}
 f(q_{wij})= &  \frac{1}{5}f(q_{wk})+\frac{2}{5}f(q_{wi})+\frac{2}{5}f(q_{wj}) \\& +\frac{1}{3} 5^{-m} \Big(\frac{7}{25}\Delta f(q_{wk})+\frac{9}{25} \Delta f(q_{wi}) + \frac{9}{25} \Delta f(q_{wj})\Big).
\end{aligned}
\end{equation*}
 Similar expressions for $f(q_{wjk})$ and $f(q_{wki})$ can be obtained.
From \cite[Lemma $5.1$]{Sahu}, we get 
\begin{equation*}
\begin{aligned}
|f(q_{wk})- f(q_{wij})|  = &\Big| f(q_{wk})- \frac{1}{5}f(q_{wk})+\frac{2}{5}f(q_{wj})+\frac{2}{5}f(q_{wi})\\ & + \frac{1}{3} 5^{-m} \Big(\frac{7}{25}\Delta f(q_{wk})+\frac{9}{25} \Delta f(q_{wi}) + \frac{9}{25} \Delta f(q_{wj})\Big)\Big| \\  \le & \Big|f(q_{wk})- \frac{1}{5}f(q_{wk})+\frac{2}{5}f(q_{wi})+\frac{2}{5}f(q_{wj})\Big| + \frac{5^{-m}}{3}K \\  \le & \Big(\frac{6}{5}\Big)^{m+1}\Big(\frac{1}{2}\Big)^{m+1} + \frac{1}{3}\Big(\frac{1}{5}\Big)^{m+1}5K,
\end{aligned}
\end{equation*}
where $K$ is a suitable constant.
Therefore, number of triangular boxes required to cover is equal to $$  \Big[\Big(\frac{6}{5}\Big)^{m+1}\Big(\frac{1}{2}\Big)^{m+1} + \frac{1}{3}\Big(\frac{1}{5}\Big)^{m+1}5K \Big]2^{m+1}= \Big(\frac{6}{5}\Big)^{m+1} + \frac{1}{3}\Big(\frac{2}{5}\Big)^{m+1}5K.$$
Now, 
\begin{equation*}
\begin{aligned}
\overline{\dim}_B(G_f) & = \lim_{\delta \to 0^+} \frac{\log N_{\delta}(G_f)}{- \log \delta} \\ & \le \lim_{m \to \infty} \frac{\log 3^{m+1}\Big[\Big(\frac{6}{5}\Big)^{m+1} + \frac{1}{3}\Big(\frac{2}{5}\Big)^{m+1}5K\Big]}{ (m+1)\log 2} \\ & = \lim_{m \to \infty} \frac{(m+1)\log 3}{(m+1) \log 2} + \lim_{m \to \infty}  \frac{\log\Big[\Big(\frac{6}{5}\Big)^{m+1} + \frac{1}{3}\Big(\frac{2}{5}\Big)^{m+1}5K\Big]}{ (m+1)\log 2} \\ & = \frac{\log 3}{\log 2} + \lim_{m \to \infty} \frac{\log \Big(\frac{6}{5}\Big)^{m+1}}{ (m+1)\log 2} +\lim_{m \to \infty}   \frac{\log\Big[1 + \frac{5K}{3}\Big(\frac{1}{3}\Big)^{m+1}\Big]}{ (m+1)\log 2} \\ & = \frac{\log(18/5)}{\log 2}.
\end{aligned}
\end{equation*}
This completes the proof.
\end{proof}
For $f:SG \rightarrow \mathbb{R}$, we define total oscillation of order $n$ by $$ R(n,f)= \sum_{w \in \{1,2,3\}^n} R_f[u_w(SG)].$$
We construct a new class of functions (see; for instance, \cite{ADBJ}) in the following way:
 $$ \mathcal{C}^{\alpha}(SG) := \{f:SG \rightarrow \mathbb{R}:~ \text{f ~is~ measurable ~and}~ \|f\|_{\mathcal{C}^{\alpha}} < \infty\} $$
 where $ 0 \le \alpha \le 1 $ and $\|f\|_{\mathcal{C}^{\alpha}}:= \sup_{n \in \mathbb{N}} \frac{R(n,f)}{2^{n(\frac{\log3}{\log2}-\alpha)}}.$
 The proof of the upcoming theorem follows on lines similar to \cite[Theorem $3.1$]{ADBJ}. However, we include the proof for reader's convenience.
\begin{theorem}\label{osc_dim_SG}
Let $f:SG \rightarrow \mathbb{R}$ be a continuous function and let $0< \gamma < 1.$ Then $\overline{\dim}_B(G_f)= 1-\gamma +\frac{\log3}{\log2} $ if and only if $f \in \Big(\cap_{\alpha < \gamma} \mathcal{C}^{\alpha}(SG)\Big)\backslash \Big(\cup_{\beta > \gamma} \mathcal{C}^{\beta}(SG)\Big).$ 
\end{theorem}
\begin{proof}
We start with $\overline{\dim}_B(G_f) = 1-\gamma +\frac{\log3}{\log2}.$  Since $\overline{\dim}_B(G_f)= 1+ \overline{ \lim}_{n \to \infty}\frac{\log R(n,f)}{n \log 2}$, for each $\epsilon >0$ we have the following: 
\begin{itemize}
\item[(1)]there exists $n_0 \in \mathbb{N}$ such that $R(n,f) \le 2^{n(\frac{\log 3}{\log 2}- \gamma+\epsilon)}$ for every $n > n_0$,
\item[(2)] a sequence $(n_k)$ with $n_k \to \infty$ and $ R(n,f) \ge 2^{n_k(\frac{\log 3}{\log 2} - \gamma - \epsilon)}.$ 
\end{itemize}
Using the boundedness of $f$ and $(1)$, we obtain $R(n,f) \le K 2^{n(\frac{\log 3}{\log 2}- \gamma+\epsilon)},~\forall ~~n \in \mathbb{N},$ where $K$ is chosen sufficiently large constant and depending on $f$. This in turn yields $f  \in \cap_{\alpha < \gamma} \mathcal{C}^{\alpha}(SG).$ Now, $(2)$ produces $ f \notin \cup_{\beta > \gamma} \mathcal{C}^{\beta}(SG).$ 
\par
To obtain the other side, we consider $f  \in \cap_{\alpha < \gamma} \mathcal{C}^{\alpha}(SG).$ That is, for each $\epsilon >0$, we have $f \in \mathcal{C}^{\gamma -\epsilon}(SG).$ More precisely, $R(n,f) \le K 2^{n(\frac{\log 3}{\log 2}- \gamma+\epsilon)}$ for every $n \in \mathbb{N}$ and for some constant $K>0.$ From the very definition of the upper box dimension, it is simple to see that $\overline{\dim}_B(G_f) \le  1-\gamma +\frac{\log3}{\log2}.$ 
Now we turn our focus on the other part, that is, $f  \notin \cup_{\beta > \gamma} \mathcal{C}^{\beta}(SG).$ In other words, for each $\epsilon >0$, $f \notin \mathcal{C}^{\gamma +\epsilon}(SG).$ Moreover, there exists a sequence $(n_k)$ depending on $\epsilon$ such that $R(n,f) \ge 2^k 2^{n_k(\frac{\log 3}{\log 2} - \gamma - \epsilon)}.$ The boundedness of $f$ provides a subsequence $( n_{k_m} )$ of $(n_k)$ such that $( n_{k_m}) \to \infty$ as $m \to \infty.$ This gives $\overline{\dim}_B(G_f) \ge 1-\gamma -\epsilon+\frac{\log3}{\log2}.$ Since $\epsilon>0$ was arbitrary real number, we therefore have $\overline{\dim}_B(G_f) \ge 1-\gamma+ \frac{\log3}{\log2}.$
\end{proof}
\begin{remark}
We have the following with respect to the end points
\[ 
\overline{\dim}_B(G_f)= \frac{\log 3}{\log 2} \iff f \in \cap_{0 < \epsilon <1} \mathcal{C}^{1 - \epsilon}(SG)
\]
and
\[ 
\overline{\dim}_B(G_f)= 1+ \frac{\log 3}{\log 2} \iff f \notin \cup_{0 < \epsilon<1} \mathcal{C}^{ \epsilon}(SG).
\]
\end{remark}
\begin{remark}
Theorem \ref{osc_dim_SG} strengthens Theorem \ref{BVSGth1}.
\end{remark}
\section{Bounded Variation}
\begin{definition}{(A)}
Let $ f: SG \rightarrow \mathbb{R}$ be a function. If a function $f$ satisfies $$\sup_{n \in \mathbb{N}} \sum_{w \in \{1,2,3\}^n} R_f[u_w(SG)] < \infty$$ then we say $f$ is of bounded variation on $SG.$ In this case, the total variation of $f$ is denoted by $V(f)= \sup_{n \in \mathbb{N}} \sum_{w \in \{1,2,3\}^n} R_f[u_w(SG)].$
\end{definition}
\begin{definition}{(B)}
Let $ f: SG \rightarrow \mathbb{R}$ be a function. If function $f$ satisfies $$\sup_{n \in \mathbb{N}} \sum_{w \in \{1,2,3\}^n} R_f[u_w(V_0)] < \infty$$ then we say $f$ is of bounded variation.
\end{definition}
\begin{example}\label{comp_exam}
The function $f :SG \to \mathbb{R}$ defined by $f(x_1,x_2)=x_1$ is a Lipschitz continuous but it is not of bounded variation with respect to the definition $(A)$ because $$ \sum_{w \in \{1,2,3\}^n} R_f[u_w(SG)] = \sum_{w \in \{1,2,3\}^n} \frac{1}{2^n}=\bigg( \frac{3}{2}\bigg)^n  \to \infty~~ \text{as}~ n \to \infty.$$
\end{example}
\begin{example}
Consider the function $f: SG \rightarrow \mathbb{R}$ defined by
 \begin{equation*}
 f(x) =
 \begin{cases}
 0 \quad \text{if} \quad x \in V_*\\
 1 \quad \text{otherwise}.
 \end{cases}
 \end{equation*}
 It is straightforward to check that the function $f$ is of bounded variation with respect to definition $(B)$ but is not of bounded variation with respect to definition $(A).$
\end{example}
\begin{example}
Consider the function $f: SG \rightarrow \mathbb{R}$ defined by
 \begin{equation*}
 f(x) =
 \begin{cases}
 1 \quad \text{if} \quad x=q_1\\
 0 \quad \text{otherwise}.
 \end{cases}
 \end{equation*}
 We see that the function $f$ is of bounded variation with respect to both definitions.
\end{example}

\begin{theorem}
The above two definitions are equivalent for the set of all continuous functions on $SG.$
\end{theorem}
\begin{proof}
Since $V_*$ is dense in SG, we have the required result.

\end{proof}

\begin{definition}{(C)}
Let $ f: SG \rightarrow \mathbb{R}$ be a function. For each partition $P: \overline{1}=w^0 \prec w^1 \prec w^2 \prec \dots \prec w^n= \overline{3}$ of $\Omega_{\mathcal{F}} ,$ we define variation of $f$ over SG as $$V(f,SG,P)= \sum_{i=1}^{n} \Big|f\big(\phi_{\mathcal{F}}(\omega^i)\big)-f\big(\phi_{\mathcal{F}}(\omega^{i-1})\big)\Big|.$$ Let us define total variation of $f$ over SG as $$V(f,SG):= \sup_{P} V(f,SG,P),$$ where the supremum is taken over all partitions $P$ of the tops code space $\Omega_{\mathcal{F}}.$ If $V(f,SG) < \infty,$ we say that $f$ is of bounded variation. The set of all functions of bounded variation on $SG$ will be denoted by $\mathcal{BV}(SG)$.
\end{definition}

\begin{remark}
Note that the space $\mathcal{BV}(SG)$ is a Banach space with respect to the norm $\|f\|:= |f(q_1)| + V(f,SG),$ where $q_1= \phi_{\mathcal{F}}(\overline{1}).$ 
\end{remark}
The next theorem shows that definition $(A)$ and definition $(C)$ are equivalent. The proof of the next theorem follows from the very construction of the Sierpi\'nski gasket, and from definitions $(A)$ and $(C)$ of bounded variation, hence omitted.
\begin{theorem}
The definitions $(A)$ and $(C)$ given above are equivalent.
\end{theorem}

\begin{remark}
In the light of Example \ref{comp_exam}, we may conjecture that every non-constant Lipschitz function is not of bounded variation with respect to $(A)$. Therefore, we need to introduce some other definitions of bounded variation on SG.
\end{remark}
 To include Lipschitz class in bounded variation class, we define another set of definitions of bounded variation on SG.
\begin{definition}{($A^*$)}
Let $ f: SG \rightarrow \mathbb{R}$ and $s=\frac{\log 3}{\log 2}.$ If function $f$ satisfies $$\sup_{n \in \mathbb{N}} \sum_{w \in \{1,2,3\}^n} \big(R_f[u_w(SG)]\big)^s < \infty$$ then we say $f$ is of bounded variation on $SG.$ In this case, the total variation of $f$ is denoted by $V^*(f)= \sup_{n \in \mathbb{N}} \sum_{w \in \{1,2,3\}^n} \big(R_f[u_w(SG)]\big)^s.$ 
\end{definition}
Similar to the above, $(B^*)$ and $(C^*)$ can be defined corresponding to $(B)$ and $(C)$ respectively.

\begin{note}\label{conj}
Every Lipschitz function is of bounded variation in the sense of $(A^*)$ on SG. Let $f:SG \to \mathbb{R}$ such that $|f(x)-f(y)| \le K\|x-y\|_2.$ Then  $$ \sum_{w \in \{1,2,3\}^n} R_f[u_w(SG)] \le \sum_{w \in \{1,2,3\}^n} \bigg( \frac{K}{2^n} \bigg)^s=\sum_{w \in \{1,2,3\}^n} \frac{K^s}{3^n} = K^s$$
holds for every $n $. Hence, $f$ is of bounded variation on SG.
This can be compared with \cite[Conjecture $5.3$]{AT}.
\end{note}
\begin{remark}
We could prove that every non-constant harmonic function $h$ is not of bounded variation with respect to the above definitions. Furthermore, we have $0< E(h) < \infty.$ Authors of \cite{AT} have obtained a similar result as above, for details, see \cite[Theorem $5.2$]{AT}. 
\end{remark}
The next theorem tells that the class of bounded variation functions in the sense of $(A^*)$ is larger than that of $(A).$
\begin{theorem}
Let $f:SG \to \mathbb{R}$. If $f$ is of bounded variation in the sense of $(A)$ then it is also bounded variation in the sense of $(A^*).$  
\end{theorem}
\begin{proof}
Since $s= \frac{\log 3}{\log 2} > 1$, the result follows immediately.
\end{proof}
To prove an analogous result of univariate case in the current study, that is, a function of bounded variation on interval can be decomposed into difference of two increasing functions, we need the following definitions.
\begin{definition}
Let $ f: SG \rightarrow \mathbb{R}$ be a function. For $x,y \in SG$ we define $x \prec y$ if $w^x \prec w^y.$ A function $f$ is said to be increasing if $f(x) < f(y)$ whenever $x \prec y.$
\end{definition}
\begin{definition}
Let $ f: SG \rightarrow \mathbb{R}$ be a function. For each $x,y \in SG$ and partition $P: w^x =w^0 \prec w^1 \prec w^2 \prec \dots \prec w^n= w^y$ of $\Omega_{[x,y]} ,$ we define variation of $f$ as $$V\Big(f,\phi_{\mathcal{F}}(\Omega_{[x,y]}),P\Big)= \sum_{i=1}^{n} \Big|f\big(\phi_{\mathcal{F}}(\omega^i)\big)-f\big(\phi_{\mathcal{F}}(\omega^{i-1})\big)\Big|.$$ Let us define total variation of $f$ over $\phi_{\mathcal{F}}(\Omega_{[x,y]})$ as $$V\Big(f,\phi_{\mathcal{F}}(\Omega_{[x,y]})\Big):= \sup_{P} V(f,[x,y],P),$$
where the supremum is taken over all partitions $P$ of $\Omega_{[x,y]} .$
\end{definition}
\begin{example}
Let $\mu$ be a Borel measure on SG. A mapping $f: SG \to \mathbb{R}$ defined by $f(x)= \mu\big(\phi_{\mathcal{F}}\big(\Omega_{[q_1,x]}\big)\big)$ and $f(q_1)=0,$ is of bounded variation on SG.
\end{example}
The proofs of the upcoming lemma and theorem are same as that of their counterpart in univariate case. Hence we omit.
\begin{lemma}
If $ f: SG \rightarrow \mathbb{R}$ is a function of bounded variation. then the function $g:SG \rightarrow \mathbb{R}$ defined by $g(x)=V\Big(f,\phi_{\mathcal{F}}(\Omega_{[q_1,x]})\Big)$ is an increasing. 
\end{lemma}
\begin{theorem}
If $ f: SG \rightarrow \mathbb{R}$ is a function of bounded variation in the sense of $(A).$ then there exist two increasing functions $f_1$ and $f_2$ such that $f=f_1 -f_2.$ Similar result holds in terms of definition $(A^*).$
\end{theorem}
\begin{theorem}\label{BVSGth3}
If $f :SG \rightarrow \mathbb{R}$ is continuous and of bounded variation in the sense of $(A).$ Then $\dim_B(G_f)= \dim_H(G_f)= \frac{\log3}{\log2}.$
\end{theorem}
\begin{proof}
In view of Lemma \ref{BVSGl1}, we have $\underline{\dim}_B G_f  \ge \dim_H G_f \ge \frac{\log3}{\log2}.$
Let $\delta= \frac{1}{2^n}$ for some $ n \in \mathbb{N}.$ From Lemma \ref{BVSGl2}, we know that the number of $\delta-$cubes that intersect the graph of $f$ is
$$N_{\delta}(G_f) \le 2.3^n + 2^n \sum_{w \in \{1,2,3\}^n} R_f[u_w(SG)].$$
Since $f$ is of bounded variation, by definition, we have $ \sum_{w \in \{1,2,3\}^n} R_f[u_w(SG)]$ is bounded for all $n \in \mathbb{N}.$
  That is, there exists $K >0$ such that $$\sum_{w \in \{1,2,3\}^n} R_f[u_w(SG)] \le K $$ for all $ n \in \mathbb{N}$.
 This in turn yields
$$ \overline{\lim}_{\delta \rightarrow 0} \frac{\log N_{\delta}(G_f)}{-\log \delta} \le \lim_{\delta \rightarrow 0}\frac{\log(2.3^n + 2^n K)}{-\log \delta} \le \frac{\log3}{\log2},$$
that is, $ \overline{\dim}_B(G_f)\le \frac{\log3}{\log2}$, completing the proof. 
\end{proof}
\begin{example}
Define a function $f$ on $SG$ as follows 

 \begin{equation*}
 f(x,y) =
 \begin{cases}
 x \sin(\frac{1}{x}) \quad \text{if} \quad x \ne 0\\
 0 \quad \text{otherwise}.
 \end{cases}
 \end{equation*}

 The function $f$ defined above is not of bounded variation in the sense of $(A).$ However, following routine calculations, we deduce  that  $\dim_B(G_f)=\dim_H(G_f)=\frac{\log3}{\log2}.$
\end{example}

  We give an example of a function which is not of bounded variation in the sense of any definition and discontinuous at each point of its domain and whose box dimension and Hausdorff dimension are $\frac{\log3}{\log2}.$
\begin{example}
Define a function $f:SG \rightarrow \mathbb{R}$ as follows
\begin{equation*}
 f(x,y) =
 \begin{cases}
 0 \quad \text{if}\quad (x,y) \in \mathbb{Q}\times \mathbb{Q}\\
 1 \quad \text{otherwise}.
 \end{cases}
 \end{equation*}

  Now, we write the graph of function $f$ as $G_f=\{(x,y,0): (x,y) \in SG  \cap (\mathbb{Q} \times \mathbb{Q}) \} \cup \{(x,y,1): (x,y) \in SG  \cap (\mathbb{Q} \times \mathbb{Q})^c \}:=G_0 \cup G_1$ The first term in the union is countable, so Hausdorff dimension of first term is zero. Using the countable stability property of Hausdorff dimension, we have $\dim_H (G_f) = \frac{\log3}{\log2}.$ We write $ \frac{\log3}{\log2}= \dim_H (G_f) \le \underline{\dim}_B (G_f) \le \overline{\dim}_B (G_f).$ Using a property of upper box dimension, we have  $\overline{\dim}_B(G_0)=\overline{\dim}_B( \overline{G_0})=\frac{\log3}{\log2}$ and $\overline{\dim}_B(G_1)=\overline{\dim}_B( \overline{G_1})=\frac{\log3}{\log2}.$ Since upper box dimension is finitely stable, we get $\overline{\dim}_B (G_f)=\frac{\log3}{\log2}.$ Therefore, $\dim_B (G_f)=\dim_H (G_f)=\frac{\log3}{\log2}.$
\end{example}

The next lemma is very useful to prove that the class of bounded variation functions is closed under addition, subtraction and multiplication.
\begin{lemma}\label{BVSGl3}
Let $f,g : SG \rightarrow \mathbb{R}$ be functions. Let $X \subseteq SG $ be a nonempty subset of the Sierpi\'nski gasket. The following inequality connects the oscillations of $f,g$ and $f+g$ over $X$;
$$ R_{f+g}[X] \le R_{f}[X]+R_{g}[X].$$
If $f$ and $g$ are bounded then 
$$ R_{fg}[X] \le M_g~ R_{f}[X]+M_f ~R_{g}[X],$$
where $M_f= \sup_{x \in X}|f(x)|$ and $M_g= \sup_{x \in X}|g(x)|.$ In particular, we have the following
\begin{itemize}
\item $V(f+g) \le V(f)+V(g)$ and $V(fg) \le M_f ~V(f)+M_g~V(g).$
\item $V^*(f+g) \le 2^{s-1}\big(V^*(f)+V^*(g)\big)$ and $V^*(fg) \le 2^{s-1}\big(M_g^s ~~V^*(f) + M_f^s ~~ V^*(g)\big).$
\end{itemize}
\end{lemma}
\begin{remark}
Comparing the above lemma to \cite[Lemma 4.14]{AT} we see that our total variation with respect to definition $(A)$ follows the expression similar to univariate real-valued case. However, the total variation defined in \cite{AT} does not follow similar expression.
\end{remark}
The following theorem follows at once from the definitions of bounded variation and Lemma \ref{BVSGl3}. Hence we omit the proof.
\begin{theorem}
The class of Bounded Variation functions in the sense of $(A)$ or $(A^*)$ is closed under addition and subtraction.
\end{theorem}
\begin{theorem}\label{BVSGth4}
The class of Bounded Variation functions in the sense of $(A)$ or $(A^*)$ is closed under multiplication.
\end{theorem}
\begin{proof}
We see that if $f$ is of bounded variation then it is bounded. Using Lemma \ref{BVSGl3}, the result follows.
\end{proof}
\begin{theorem}\label{BVSGth5}
The class of Bounded Variation functions in the sense of $(A)$ is closed under division provided denominator bounded away from zero. Similar result holds in terms of definition $(A^*).$
\end{theorem}
\begin{proof}
In the light of Theorem \ref{BVSGth4}, it suffices, for the first statement, to consider the case of $ \frac{1}{f}$ for $f$ in the class of bounded variation and $|f| \ge m > 0.$\\
Let $M$ be the total variation of $f$, and for each $n$ let $N_n$ be the number of cells in the net of $3^n$ cells, in which $f$ changes sign; then
\begin{equation}\label{useith}
 M \ge \sum_{w \in \{1,2,3\}^n} R_f[u_w(SG)] \ge 2m N_n.
 \end{equation}
Let us set $$ \sum_{w \in \{1,2,3\}^n} R_{1/f}[u_w(SG)]= {\Sigma}' + {\Sigma }'' ,$$
where $ \Sigma'$ representing the sum over the cells in which $f$ changes sign and $ \Sigma''$ the sum over the remaining cells. In each cell of the first set, we have $ R_{1/f}[u_w(SG)] \le \frac{2}{m}.$ We denote the least upper bound and greatest lower bound of $|f|$ in the $u_w(SG)$ by $M_w$ and $m_w$ respectively. Now, for each triangular cell of the second set, we get $$ R_{1/f}[u_w(SG)] = \frac{1}{m_w}- \frac{1}{M_w}\le \frac{(M_w- m_w)}{m^2}= \frac{R_f[u_w(SG)]}{m^2}.$$
Using \ref{useith}, we obtain 
$$ \sum_{w \in \{1,2,3\}^n} R_{1/f}[u_w(SG)] \le \frac{M}{m^2}+ \frac{2 N_n}{m} \le \frac{2M}{m^2}$$ for all $n \in \mathbb{N}.$ Hence the proof is complete.
We note that the similar technique will work for proving the result in terms of definition $(A^*).$

\end{proof}

\begin{theorem}\label{BVSGth6}
If $f: SG \rightarrow \mathbb{R}$ is of bounded variation on $SG$ in the sense of $(A)$ or $(A^*)$ then $f$ is continuous almost everywhere in the sense of $\frac{\log3}{\log2}-$dimensional Hausdorff measure.
\end{theorem}
\begin{proof}
Let $\epsilon > 0.$ Define $X_{\epsilon} = \{ x \in SG: \text{points at which f has a saltus} \ge \epsilon \}.$ Assume $\frac{\log3}{\log2}$- dimensional Hausdorff measure of $X_{\epsilon}$ is positive, that is, $\mathcal{H}^s(X_{\epsilon}) > 0,$ where $s=\frac{\log3}{\log2}.$ Let the $\frac{\log3}{\log2}$-dimensional Hausdorff measure of SG be denoted by $\mathcal{H}^s(SG).$ For a triangular net of $3^n$ cells, we observe that at least $\left \lceil {\frac{3^n \mathcal{H}^s(X_{\epsilon}) }{\mathcal{H}^s(SG)}}\right \rceil$ cells of the triangular net must contain points of $X_{\epsilon}$, where $\left \lceil {.}\right \rceil $ denotes the ceiling function. Hence, we obtain $$    \sum_{w \in \{1,2,3\}^n} R_f[u_w(SG)] \ge \left \lceil {\frac{3^n \mathcal{H}^s(X_{\epsilon}) }{\mathcal{H}^s(SG)}}\right \rceil \epsilon ,$$ 
which is unbounded unless $ \mathcal{H}^s(X_{\epsilon})$ is zero. Therefore, if $f$ is of bounded variation in the sense of $(A)$, $ \mathcal{H}^s(X_{\epsilon})$ must vanish for every $ \epsilon > 0$, and by a classical argument, it follows that the discontinuities of $f$ are a set of $\frac{\log3}{\log2}$-dimensional Hausdorff zero measure. Similarly, we obtain the result for $(A^*).$ 
\end{proof}
\begin{theorem}
If $f:SG \rightarrow \mathbb{R}$ is continuous and of bounded variation on $SG$ in the sense of $(A)$ or $(A^*)$ then $ 0< \mathcal{H}^s(G_f)<\infty,$ where $s =\frac{\log3}{\log2}.$ In particular, $\dim_H(G_f)= s.$
\end{theorem}
\begin{proof}
We only prove the result in terms of definition $(A)$ because for $(A^*)$ the result follows immediately on similar lines. 
Let us first define a mapping $\mathcal{T}_f : G_f \rightarrow SG $ by $\mathcal{T}_f((t,f(t)))=t.$ Then
$$\|\mathcal{T}_f((t,f(t)))-\mathcal{T}_f((u,f(u)))\|_2 = \|t -u\|_2 \le \|(t,f(t))-(u,f(u))\|_2.$$ Therefore, $\mathcal{T}_f$ is a Lipschitz map. Using a properties of Hausdorff measure, we have $\mathcal{H}^s (\mathcal{T}_f(G_f)) \le \mathcal{H}^s (G_f).$ It can be straightforwardly checked that the mapping $\mathcal{T}_f$ is onto. Hence $\mathcal{H}^s (G_f)>0.$
\par
 
Now, using a natural covering, we have 
\begin{equation}
\begin{aligned}
\sum_{w \in \{1,2,3\}^n} |F_w|^s & \le \sum_{w \in \{1,2,3\}^n} 2^{s/2} \Big(\max\Big\{\frac{1}{2^n}, R_f[u_w(SG)]\Big\}\Big)^s \\ & = \sum_{w \in \{1,2,3\}^n} 2^{s/2} \max\Big\{\Big(\frac{1}{2^n}\Big)^s, (R_f[u_w(SG)])^s\Big\} \\ & =\sum_{w \in \{1,2,3\}^n} 2^{s/2} \max\Big\{\frac{1}{3^n}, (R_f[u_w(SG)])^s\Big\} \\ & \le \sum_{w \in \{1,2,3\}^n} 2^{s/2} \max\Big\{\frac{1}{3^n}, R_f[u_w(SG)]\Big\},
\end{aligned}
\end{equation}
where $F_w= u_w(SG) \times R_f[u_w(SG)].$ Since both $\sum_{n=1}^{\infty} \frac{1}{3^n}$ and $\sup_{n \in \mathbb{N}} R_f[u_w(SG)]$ are finite, we deduce that $\sum_{w \in \{1,2,3\}^n} |F_w|^s < \infty .$ From the hypothesis of bounded variation, the quantity $\max\{\frac{1}{3^n}, R_f[u_w(SG)]\}$ can be made arbitrarily small. That is, for any $\delta >  0,$ one can choose $n$ large enough such that 
\[
|F_w| \le 2^{s/2} \max\Big\{\frac{1}{3^n}, R_f[u_w(SG)]\Big\} \le \delta.
\]
 Therefore, $\mathcal{H}^s_{\delta}(G_f) \le \sum_{w \in \{1,2,3\}^n} |F_w|^s < \infty.$ Consequently, 
$\mathcal{H}^s (G_f)< \infty.$
\end{proof}

\section{Conclusion}
We introduced a new definition of bounded variation on the Sierpi\'nski gasket. In the light of Note \ref{conj}, we claim that our approach is different from that of \cite{AT}, and also very useful in applications. In the paper, we proved various properties of a bounded variation function analogous to univariate real-valued case. We believe that our definition can be extended to more general class of fractals such as finitely ramified cell structure, see, for instance, \cite{Tep}. 

 \subsection*{Acknowledgements}
  The first author expresses his gratitude to the University Grants
     Commission (UGC), India for financial support.

\bibliographystyle{amsplain}

\end{document}